\documentclass[12pt]{article}
\usepackage{amsmath,amsfonts,amssymb,amsthm}
\usepackage{txfonts}
\usepackage{url}

\newtheorem{Theorem}{Theorem}[section] 
\newtheorem{Prop}[Theorem]{Proposition}
\newtheorem{Cor}[Theorem]{Corollary}
\newtheorem{Lemma}[Theorem]{Lemma}
\newtheorem{Rem}[Theorem]{Remark}

\newtheorem{Ex}[Theorem]{Example}
\numberwithin{equation}{section}

\def\N{\mathbb{N}}
\def\R{\mathbb{R}}
\def\Q{\mathbb{Q}}
\def\P{\mathbb{P}}
\def\E{\mathbb{E}}
\def\X{\mathbb{X}}
\def\XX{\mathcal{X}}
\def\Y{\mathbb{Y}}
\def\YY{\mathcal{Y}}

\def\O{\mathcal{O}}
\def\F{\mathcal{F}}
\def\Cov{\mathrm{Cov}}
\def\hs{\hspace{1pt}}
\def\I{\mathbf{1}}
\allowdisplaybreaks

\begin{document}

\title{Concentration inequalities\\
for measures of a  Boolean model}

\date{\today}
\renewcommand{\thefootnote}{\fnsymbol{footnote}}

\author{Fabian Gieringer\footnotemark[1]\, and G\"unter Last\footnotemark[1]}

\footnotetext[1]{fabian.gieringer@kit.edu,  guenter.last@kit.edu,  
Karlsruhe Institute of Technology, Institute of Stochastics,
76128 Karlsruhe, Germany. }

\maketitle

\begin{abstract}
\noindent We consider a Boolean model $Z$ driven by a Poisson particle
process $\eta$ on a metric space $\Y$. We study the random variable
$\rho(Z)$, where $\rho$ is a (deterministic) measure on $\Y$. 
Due to the interaction of overlapping particles, the
distribution of $\rho(Z)$ cannot be described explicitly.
In this note we derive concentration inequalities for $\rho(Z)$.
To this end we first prove two concentration inequalities  for 
functions of a Poisson process on a general phase space. 
\end{abstract}



\section{Introduction}

Let $\Y$ be a locally compact separable metric space 
and let $\F$ be the space of all closed
subsets of $\Y$ equipped with a suitable $\sigma$-field.
Let $\eta$ be a Poisson process on $\F$ with a $\sigma$-finite
intensity measure $\Lambda$. If $\eta\{K\}>0$, then we write
$K\in\eta$ and say that $K$ is a {\em particle} of $\eta$.
The {\em Boolean model}
associated with $\eta$ is the random set $Z$ defined by
the union of all particles, that is
\begin{align*}
Z:=\bigcup_{K\in\eta} K.
\end{align*}
Let $\rho$ be a measure on $\Y$ satisfying
\begin{align*}
\int_\F \rho(K)\,\Lambda(dK)<\infty.
\end{align*}
Then $\rho(Z)$ is a finite random variable even though $Z$ might not be
a random closed set in the sense of \cite{Mat,SW}.

The random set $Z$ is a fundamental model of stochastic geometry and continuum percolation;
see \cite{SKM, Mat, SW}. Explicit formulae for the distribution
of geometric functionals of the Boolean model are not available, even
not in the simplest case of a stationary Boolean on $\R^d$ and
$\rho=\lambda_d(\cdot\cap W)$ being the restriction of Lebesgue measure to a convex and compact
set $W$. The reason for the absence of such formulae is the interaction between the
particles from $\eta$ caused by overlapping. 
One way out are moment formulae and central limit theorems; see e.g.\ \cite{HLS}
and \cite[Chapter 22]{LastPenrose17}.
In this paper we will prove {\em concentration inequalities} of the form
\begin{align*}
\P(F-\E[F]\ge r)\le\exp\Big(\inf_{s\ge 0}\int_0^s \varv(u)\,du-sr\Big),\quad r>0,
\end{align*}
where the function $\varv\colon[0,\infty)\to[0,\infty]$ is determined
by $\Lambda$ and $\rho$. In the stationary Euclidean case 
such inequalities were first proved in \cite{Hei}.
Our bounds improve these results. Moreover, we generalize the setting of \cite{Hei}
in several ways. First, we study the Boolean model on a metric space $\Y$ and not only
on $\R^d$. Second, we will allow that compact subsets of $\Y$ are intersected
by infinitely many Poisson particles. Hence, in general, the random set $Z$ is not closed
and its boundary might have fractal 
properties. 
Roughly speaking, this means that we can allow for a $\sigma$-finite
distribution of the typical grain. Closely related models of this type were introduced
in \cite{Zaehle84}, a seminal paper on fractal percolation,
that was almost completely ignored  in the later literature.
Third, we consider general measures and not only the volume.
Finally, our method allows to treat also Lipschitz functions of these
measures.

Similarly as in \cite{Hou,HP}
our approach is based on a covariance identity for square integrable
Poisson functionals. In fact we first prove a concentration inequality  for 
functions of a Poisson process on a general phase space. 
Using the log-Sobolev inequality, related concentration inequalities  
were derived in \cite{BLM03,BP,Wu}. 

\section{Concentration of Poisson functionals}\label{general}

Let $(\X,\XX)$ be a measurable space and let $\Lambda$
be a $\sigma$-finite measure on $\X$. Let $\eta$ be a Poisson process
on $\X$ with intensity measure $\Lambda$, 
defined over a probability space $(\Omega,\mathcal{A},\P)$; see \cite{LastPenrose17}.
In particular, $\eta$ is a point process, that is a measurable
mapping from $\Omega$ to the space $\mathbf{N}=\mathbf{N}(\X)$
of all $\sigma$-finite measures with values in $\bar{\N}_0:=\{\infty,0,1,2,\ldots\}$,
where $\mathbf{N}$ is equipped with the smallest $\sigma$-field
$\mathcal{N}$ such that $\mu\mapsto\mu(B)$ is measurable for all $B\in\XX$.
The distribution of $\eta$ is denoted by $\Pi_{\Lambda}:=\P(\eta\in\cdot)$.
Since we are only interested in distributional properties of $\eta$,
Corollary 6.5 in \cite{LastPenrose17} shows that it is no restriction
of generality to assume that $\eta$ is {\em proper}. This means  
that there exist random elements
$X_1,X_2,\ldots$ in $\X$ and an $\bar{\N}_0$-valued random
variable $\kappa$ such that almost surely
$\eta=\sum_{n=1}^\kappa\delta_{X_n}$.

Let $0\le t\le1$ and $Y_1,Y_2,\ldots$ be a sequence of independent
random variables with distribution $(1-t)\delta_0+t\delta_1$,
independent of $\eta$. Define
$\eta_t:=\sum_{n=1}^\kappa Y_n\delta_{X_n}$ as the {\it $t$-thinning} of $\eta$.
Then $\eta_t$ and $\eta-\eta_t$ are independent Poisson processes
with intensity measures $t\Lambda$ and $(1-t)\Lambda$, respectively. 
Given $x\in\X$ and a measurable function
$f\colon\mathbf{N}\to\R$, the {\it difference operator} $D_xf$ is defined
by
$$
D_xf(\mu):=f(\mu+\delta_x)-f(\mu),\quad \mu\in\mathbf{N}.
$$
This mapping is measurable since $(\mu,x)\mapsto\mu+\delta_x$ is
measurable. We call a random variable $F$ a {\it Poisson functional}
if there is a measurable $f\colon\mathbf{N}\to\R$ such that $F=f(\eta)$
almost surely. 
In this case we define
$$
D_xF:=D_xf(\eta),\quad x\in\X,
$$
(which is almost surely, for $\Lambda$-almost all $x$, independent of
the choice of an admissible $f$) and further a mapping
$DF\colon\Omega\times\X\to\R$, given by
$(\omega,x)\mapsto(D_xF)(\omega)$.

The starting point of our concentration inequalities is the following
covariance identity; see Theorem 20.2 in \cite{LastPenrose17}. The conditional
expectation appearing there can be dropped such that we get the
following identity. 

\begin{Prop}
Let $F=f(\eta)$ and $G=g(\eta)$ be Poisson functionals
such that $F,G\in L^2(\P)$ and $DF,DG\in L^2(\P\otimes\Lambda)$. Then
\begin{align}\label{193}
\Cov(F,G)=\E\Big[\int_{\X}D_xF\int_0^1\int_\mathbf{N}D_xg(\eta_t+\mu)\,
\Pi_{(1-t)\Lambda}(d\mu)\,dt\,\Lambda(dx)\Big].
\end{align}
\end{Prop}

For $F=f(\eta)\in L^2(\P)$ we define
\begin{equation*}
s_F:=\sup\big\{s\ge0:e^{sF}\in L^2(\P), De^{sF}\in L^2(\P\otimes\Lambda)\big\},
\end{equation*}
where the case $s_F=\infty$ is possible.
Define
\begin{equation*}
V_F(s):=\int_{\X}\big(e^{sD_x F}-1\big)\int_0^1\int_{\mathbf{N}}D_xf(\eta_t+\mu)\,
\Pi_{(1-t)\Lambda}(d\mu)\,dt\,\Lambda(dx),\quad 0\le s<s_F.
\end{equation*}

The following bound for the cumulant-generating function is the
main result of this section.

\begin{Theorem}\label{main} Let $F=f(\eta)\in L^2(\P)$ and $s\in(0,s_F)$. Then
\begin{align}\label{haesslich}
  \log\E\big[e^{s(F-\E[F])}\big] 
&\le\inf_{0<\theta<1}\frac{\theta}{1-\theta}
\Big(\int_0^s\Big(\frac1s\log\E\big[e^{sV_F(u)/\theta}\big]-\frac{1}{u}\log\E\big[e^{uF}\big]\Big)\,du+s\hs\E[F]\Big)\\
&\le\inf_{0<\theta<1}\frac{\theta}{s(1-\theta)}\int_0^s\log\E\big[e^{sV_F(u)/\theta}\big]\,du.\label{schoen}
\end{align}
\end{Theorem}
\begin{proof} We combine the idea of the proof of Lemma 3.1 in \cite{HP}
(see also the proof of Theorem 1 in \cite{Hou}) with Lemma 11 in Massart \cite{Mas}.
Let $\theta\in(0,1)$ and $s\in(0,s_F)$ be such that
$\E\big[e^{sV_F(u)/\theta}\big]<\infty$ and let $u\in(0,s]$. Since
$u<s_F$, we can use the covariance identity \eqref{193} to obtain that
\begin{align*}
  \Cov\big(F,e^{u F}\big)&=\E\Big[\int_{\X}\int_0^1\big(D_xe^{u F}\big)\int_{\mathbf{N}}D_xf(\eta_t+\mu)\,\Pi_{(1-t)\Lambda}(d\mu)\,dt\,\Lambda(dx)\Big]=\E\big[e^{uF}\hs
  V_F(u)\big].
\end{align*}
Now, Lemma 11 of Massart \cite{Mas} applied to $V_F(u)/\theta$ and $F$ yields
\begin{align*}
  \frac{\E\big[e^{uF}\hs
    V_F(u)\big]}{\E\big[e^{uF}\big]}
&\le\frac{\theta\hs\E\big[e^{uF} F\big]}{\E\big[e^{uF}\big]}+\frac{\theta}{u}\log\E\big[e^{uV_F(u)/\theta}\big]
-\frac\theta  u\log\E\big[e^{uF}\big].
\end{align*}
The combination of the last two displays leads to the inequality
\begin{equation*}
  \frac{\E\big[F\hs e^{uF}\big]}{\E\big[e^{uF}\big]}-\E[F]
=\frac{\Cov(F,e^{u F})}{\E\big[e^{uF}\big]}\le\frac{\theta\hs\E\big[e^{uF}\hs F\big]}{\E\big[e^{uF}\big]}
+\frac\theta u\log\E\big[e^{uV_F(u)/\theta}\big]-\frac\theta u\log\E\big[e^{uF}\big]
\end{equation*}
and a simple rearrangement yields
\begin{align*}
  \frac{\E\big[Fe^{uF}\big]}{\E\big[e^{uF}\big]}&\le\frac{\theta}{u(1-\theta)}
\Big(\frac{u}{\theta}\hs\E[F]+\log\E\big[e^{uV_F(u)/\theta}\big]-\log\E\big[e^{uF}\big]\Big).
\end{align*}
Setting $h(t):=\log\E\big[e^{tF}\big]$ and $g_u(t):=\log\E\big[e^{tV_F(u)}\big]$, $t\ge0$, we have
\begin{align*}
h(s)=h(0)+\int_0^sh'(u)\,du&=\int_0^s\frac{\E\big[F\hs e^{uF}\big]}{\E\big[e^{uF}\big]}\,du\\
&\le \int_0^s\frac{\theta}{u(1-\theta)}\Big(\frac{u}{\theta}\hs\E[F]+g_u(u/\theta)-h(u)\Big)\,du.
\end{align*}
By $g_u(0)=0$ and the convexity of $g_u$, we have 
$g_u(\frac{u}{s}\hs t)\le \frac{u}{s}\hs g_u(t)$ for $t>0$, thus
\begin{align*}
h(s)&\le\frac{s}{1-\theta}\E[F]+\frac{\theta}{1-\theta}\int_0^s\Big(\frac1sg_u(s/\theta)-\frac{1}{u}h(u)\Big)\,du.
\end{align*}
From $\log\E\big[e^{s(F-\E[F])}\big]=h(s)-s\hs\E[F]$ and the preceding
inequality, \eqref{haesslich} follows. Using Jensen's inequality, this
simplifies to \eqref{schoen}.
\end{proof}

Theorem \ref{main} and the well-known Chernoff bound (see \cite{Che}) 
\begin{align}\label{Chernoff}
\P(F-\E[F]\ge r)\le\inf_{s>0}\frac{\E[e^{s(F-\E[F])}]}{e^{sr}},\quad r\ge 0,
\end{align}
(a direct consequence of Markov's inequality) imply a concentration inequality.
If $V_F(\cdot)$ has a deterministic bound, 
this inequality can be simplified as follows.

\begin{Cor}\label{maincor} Let $F=f(\eta)\in L^2(\P)$ and
assume that $\varv\colon[0,s_F)\to\R$ is a measurable function
such that almost surely $V_F(s)\le \varv(s)$ for each $s\in[0,s_F)$. Then,
\begin{equation*}
\P(F-\E[F]\ge r)\le\exp\Big(\inf_{0<s<s_F}\int_0^s\varv(u)\,du-sr\Big),\quad r\ge0.
\end{equation*}
\end{Cor}
\begin{proof}
Let $r\ge0$ and $s\in (0,s_F)$. By the Chernoff bound \eqref{Chernoff}, inequality \eqref{schoen} 
and assumption $V_F(s)\le \varv(s)$, we get
\begin{align*}
  \P(F-\E[F]\ge  r)
&\le\E\big[e^{s(F-\E[F])}\big]e^{-sr}\le\exp\Big(\inf_{0<\theta<1}\frac{1}{1-\theta}\int_0^s\varv(u)\,du-sr\Big).
\end{align*}
We have $\int_0^s\varv(u)\,du\ge0$ since the contrary would lead to
$\P(F-\E[F]\ge 0)\le 0$ which is obviously wrong. 
Since $\inf_{0<\theta<1}(1-\theta)^{-1}=1$, we obtain that
\begin{equation}\label{foreachs}
\P(F-\E[F]\ge r)\le\exp\Big(\int_0^s\varv(u)\,du-sr\Big),\quad s\in[0,s_F),
\end{equation}
and hence the assertion.
\end{proof}

\begin{Rem}
\rm Concentration inequalities for the lower tail can be derived analogously. Under the obvious integrability assumptions, the bounds (\ref{haesslich}) and (\ref{schoen}) hold again upon replacing $F$ by $-F$ and $V_F$ by $V_{-F}$. Thus, by the Chernoff bound $\P(F-\E[F]\le -r)\le\inf_{s>0}\E[e^{-s(F-\E[F])}]e^{-sr}$, $r\ge0$, a result analogous to Corollary \ref{maincor} gives a bound for the lower tail when $V_{-F}$ has a deterministic bound. Hence, all results relying on Corollary \ref{maincor} can be given for the lower tail as well. 
\end{Rem}


Our next result was motivated by a question in \cite{BP} whether the
Mecke formula (cf.\ \cite{LastPenrose17}) can be combined with
the covariance identity to yield reasonable concentration inequalities.

\begin{Theorem}\label{thmecke}
Let $F=f(\eta)\in L^2(\P)$ be such that $DF\ge0$ holds $(\P\otimes\Lambda)$-almost
everywhere. Assume further that there exist 
a measurable function $g\colon \X\to[0,\infty)$
and constants $a>0$ and $b\ge0$ such that a.s.
\begin{equation}
  \int_0^1\int_\mathbf{N}D_xf(\eta_t+\mu)\,\Pi_{(1-t)\Lambda}(d\mu)\,dt\le g(x),
\quad\Lambda\text{-a.e.\ $x\in\X$},\label{c1as}
\end{equation}
and 
\begin{align}\label{e2.10}
\int D_xf(\eta-\delta_x)g(x)\,\eta(dx)\le aF+b.
\end{align}
Then 
\begin{equation}\label{toopt}
  \P(F-\E[F]\ge r)\le\exp\bigg(\inf_{0<s<a^{-1}\wedge s_F}\bigg(-\frac{bs}{a}
-s\hs\E[F]-\bigg(\frac{\E[F]}{a}+\frac{b}{a^2}\bigg)\log(1-as)-sr\bigg)\bigg).
\end{equation}
In particular, if $a^{-1}\le s_F$, we have
\begin{align}\label{e2.11}
  \P(F-\E[F]\ge
  r)&\le\exp\bigg(-\frac1{a}\bigg[r+\bigg(\E[F]+\frac{b}{a}\bigg)
\log\bigg(\frac{\E[F]+b/a}{r+\E[F]+b/a}\bigg)\bigg]\bigg).
\end{align}
\end{Theorem}
\begin{proof}
Let $0<u\le s<a^{-1}\wedge s_F$. By the covariance identity \eqref{193} and assumption \eqref{c1as}, we have
\begin{align*}
\Cov\big(F, e^{u F}\big)\le \hs\E\Big[\int D_xe^{u F}g(x)\,\Lambda(dx)\Big].
\end{align*}
Applying the Mecke formula and the elementary bound $e^z-1\le ze^z$, $z\in\R$, yields
\begin{align*}
  \Cov\big(F, e^{u F}\big)&\le \hs\E\Big[\int D_xe^{u f(\eta-\delta_x)}g(x)\,\eta(dx)\Big]
= \hs \E\Big[\int e^{u f(\eta-\delta_x)}\big(e^{u D_xf(\eta-\delta_x)}-1\big)g(x)\,\eta(dx)\Big]\\
&\le u\,\E\Big[\int e^{u f(\eta)} D_xf(\eta-\delta_x)g(x)\,\eta(dx)\Big].
\end{align*}
Assumption \eqref{e2.10} yields that 
$\Cov\big(F, e^{u F}\big)\le ua\hs\E\big[F\hs e^{u F}\big]+ub\hs\E\big[e^{uF}\big]$. 
By a rearrangement, we obtain
\begin{align*}
\frac{\E[F\hs e^{uF}]}{\E[e^{uF}]}\le\frac{\E[F]+ub}{1-au}.
\end{align*}
Setting $h(u):=\log\E[e^{uF}]$, we get $h'(u)\le(\E[F]+ub)/(1-au)$ and
consequently, by integration and $h(0)=0$, the bound
$\log\E[e^{sF}]\le-bs/a-(\E[F]/a+b/a^2)\log(1-as)$.  Using the
Chernoff bound \eqref{Chernoff}, we obtain (\ref{toopt}). The second
assertion (\ref{e2.11}) then follows by optimising with
$s=r/(ar+a\E[F]+b)$. This choice of $s$ is at most $a^{-1}$, since
$F+b/a\ge0$ a.s. by assumption (\ref{e2.10}), and the case
$\E[F]+b/a=0$ can be ruled out as this implies $DF\equiv0$ and
therefore a trivial concentration.
\end{proof}

\section{General Boolean models}\label{sgeneralBoolean}

In this section we consider a locally compact separable Hausdorff space $\Y$ together
with the Borel $\sigma$-field $\YY$. Let $\F\equiv\F(\Y)$ denote 
the class of closed subsets of $\Y$
equipped with the Fell topology generated by the
sets $\{F\in\F:F\cap C=\emptyset\}$ and $\{F\in\F:F\cap
G\neq\emptyset\}$ for arbitrary compact sets $C\subset\Y$ and open
sets $G\subset\Y$; see \cite{Mat,SW}.
The associated Borel $\sigma$-field is denoted by $\mathcal{B}(\F)$.

Let $\rho$ be a 
measure on $\Y$ and let $\F'\in\mathcal{B}(\F)$ be such that
$\rho$ is finite on $\F'$ and
$K\mapsto \rho(K)$ is measurable on $\F'$. We also assume that $\F'$ is closed
under (finite) unions  
and equip $\F'$ with the trace $\sigma$-field 
$\{B\cap \F':B\in\mathcal{B}(\F)\}$.

Let $\Lambda$ be a $\sigma$-finite intensity measure on $\F'$ satisfying 
\begin{align}\label{intas1} 
\int \rho(K)\,\Lambda(dK)<\infty.
\end{align}
Let $(B_n)_{n\in\N}$ be such that $B_n\uparrow\F'$ and
$\Lambda(B_n)<\infty$ for all $n\in\N$. Let $\mathbf{N}_l$ denote
the measurable set of all $\mu\in \mathbf{N}=\mathbf{N}(\F')$
satisfying $\mu(B_n)<\infty$ for each $n\in\N$. 
Note that $\P(\eta\in\mathbf{N}_l)=\Pi_\Lambda(\mathbf{N}_l)=1$.
Define
\begin{equation}\label{Zdef}
Z(\mu):=\bigcup_{K\in\hs\mu}K,\quad\mu\in\mathbf{N}_l,
\end{equation}
and $Z(\mu):=\emptyset$ for
$\mu\in\mathbf{N}\setminus\mathbf{N}_l$. 
The random set $Z:=Z(\eta)$ is called the {\it Boolean model}
governed by $\eta$. 
Now consider the function $f\colon\mathbf{N}\to[0,\infty)$ given
by
$$
f(\mu):=\I\{\rho(Z(\mu))<\infty\}\rho(Z(\mu)),
$$
with the convention $0\cdot\infty:=0$.
Our goal is to obtain a concentration inequality for 
\begin{equation*}
F:=f(\eta).
\end{equation*}
Campbell's formula (Proposition 2.7 in \cite{LastPenrose17}) and assumption \eqref{intas1} 
show that $\int\rho\,d\eta<\infty$ a.s., so that 
the sub-additivity of $\rho$ shows that
$\rho(Z)<\infty$ a.s. In particular $\P(F=\rho(Z))=1$.
However, we need to check that $\rho(Z)$ is a random variable.
This follows from 
the assumption on $\rho$ and the next lemma.

\begin{Lemma}\label{musum} 
The mapping $(x,\mu)\mapsto \I\{x\in Z(\mu)\}$ is measurable
on $\Y\times \mathbf{N}$.
Furthermore, for each $K\in\mathcal{Y}$ with $\rho(K)<\infty$, the mapping 
$\mu\mapsto \rho(Z(\mu)\cap K)$ is measurable on $\mathbf{N}$.
Finally $\mu\mapsto \rho(Z(\mu))$ is measurable on $\mathbf{N}$.
\end{Lemma}
\begin{proof} By Theorem 1-2-1 in \cite{Mat}, $\F$ is a compact
and separable Hausdorff space and hence $\F'$ (equipped with the trace
$\sigma$-field) is a Borel space.
By \cite[Proposition 6.2]{LastPenrose17} (see also the proof of
\cite[Proposition 6.3]{LastPenrose17}), there exist
measurable functions $\pi_n\colon\mathbf{N}\to\F'$ such that
$$
\mu=\sum_{n=1}^{\mu(\F')}\delta_{\pi_n(\mu)},\quad \mu\in\mathbf{N}_l.
$$
This shows that $Z(\mu)=\cup^{\mu(\F')}_{n=1}\pi_n(\mu)$ for each $\mu\in\mathbf{N}_l$.
By \cite[Theorem 2-5-1]{Mat}, the mapping $(K,x)\mapsto \I\{x\in K\}$
is measurable on $\F\times\Y$. Since
\begin{align*}
\I\{x\notin Z(\mu)\}=\prod_{n=1}^{\mu(\F')}\I\{x\notin\pi_n(\mu)\},\quad \mu\in\mathbf{N}_l,
\end{align*}
this proves the first assertion. The second assertion follows from
Fubini's theorem.

By monotone convergence, the third assertion follows, once we have shown
that
\begin{align*}
\mu\mapsto \rho \bigg(\bigcup^{\min\{\mu(\F'),n\}}_{m=1}\pi_m(\mu)\bigg)
\end{align*}
is measurable for each $n\in\N$. By  \cite[Corollary 1-2-1]{Mat}
the mapping $(L,L')\mapsto L\cup L'$ is measurable
on $\F\times\F$ and hence also on $\F'\times\F'$. Since $\F'$ is closed
under unions, it follows that 
$\mu\mapsto \bigcup^{\min\{\mu(\F'),n\}}_{m=1}\pi_m(\mu)$ is a measurable mapping
from $\mathbf{N}_l$ to $\F'$. Since 
$L\mapsto \rho(L)$ is measurable
on $\F'$, the final assertion now follows.
\end{proof}

Let $t\in[0,1]$. We now compute the probability that a point $y\in\Y$
lies inside the Boolean model of intensity measure $t\Lambda$. We use
the notation 
\begin{align}\label{e3.4}
\F'_y:=\{K\in\F':y\in K\},\quad y\in\Y.
\end{align}
Since $(K,y)\mapsto\I\{y\in K\}$ is $\mathcal{B}(\F)\otimes\mathcal{Y}$-measurable, this is a measurable set.
The first definition in
\eqref{Zdef} and the defining properties of a Poisson process yield
\begin{align}\label{volfrac}
  \int\I\{y\in  Z(\mu)\}\,\Pi_{t\Lambda}(d\mu)&=1-\P(\eta_{t}(\F'_y)=0)
=1-\exp\big(-t\Lambda\big(\F'_y\big)\big).
\end{align}

\begin{Lemma}\label{l3.2} We have that
\begin{align*}
D_Kf(\mu)=\rho(K)-\rho(Z(\mu)\cap K),\quad \Pi_\Lambda\otimes\Lambda\text{-a.e.\ $(\mu,K)$}.
\end{align*}
\end{Lemma}
\begin{proof}
For each $\mu\in\mathbf{N}_l$ and each $K\in\F'$ we have that
\begin{align*}
D_Kf(\mu)=\rho(Z(\mu+\delta_K))-\rho(Z(\mu))=\rho(Z(\mu)\cup K))-\rho(Z(\mu)).
\end{align*}
Since $\Pi_\Lambda(\{\mu\in\mathbf{N}_l:\rho(Z(\mu))<\infty\})=1$ 
we can use the additivity of $\rho$ to conclude the proof. 
\end{proof}

\begin{Lemma}\label{difflem}
Let $t\in[0,1]$ and $K\in\F'$. Then $\P$-a.s.
$$
\int D_Kf(\eta_t+\mu)\,\Pi_{(1-t)\Lambda}(d\mu)
\le\int\I\{y\in K\}\exp\big(-(1-t)\Lambda\big(\F'_y\big)\big)\,\rho(dy).
$$
\end{Lemma}
\begin{proof} It follows from Lemma \ref{l3.2} and the superposition
theorem for Poisson processes that $\P$-a.s.
\begin{align*}
D_Kf(\eta_t+\mu)\le \rho(K)-\rho(Z(\mu)\cap K),\quad \Pi_{(1-t)\Lambda}\otimes\Lambda\text{-a.e.\ $(\mu,K)$}.
\end{align*}
Furthermore,
\begin{align*}
\int\rho(Z(\mu)\cap K)\,\Pi_{(1-t)\Lambda}(d\mu)
&=\iint\I\{y\in Z(\mu)\}\I\{y\in K\}\,\Pi_{(1-t)\Lambda}(d\mu)\,\rho(dy)\\
&=\rho(K)-\int\I\{y\in K\}\exp\big(-(1-t)\Lambda\big(\F'_y\big)\big)\,\rho(dy),
\end{align*}
where we have used \eqref{volfrac}.
The assertion now follows.
\end{proof}
Define
\begin{align*}
\rho^*(K):=\int_K\tau\big(\Lambda(\F'_y)\big)\,\rho(dy),\quad K\in\mathcal{F}',
\end{align*}
where $\tau\colon[0,\infty]\to[0,1]$ is given by 
$\tau(z):=(1-e^{-z})/z$ for $z\in(0,\infty)$, $\tau(0):=\lim_{z\downarrow0}\tau(z)=1$
and $\tau(\infty):=\lim_{z\to\infty}\tau(z)=0$.
We define a measure $\nu^*$ on $[0,\infty)$ by
\begin{align}\label{levyanalog}
\nu^*:=\int_{\F'}\I\{\rho(K)\in\cdot\}\frac{\rho^*(K)}{\rho(K)}\,\Lambda(dK),
\end{align}
with the convention $0/0:=0$ and another measure $\nu$ 
on $[0,\infty)$ by
\begin{align*}
\nu:=\int_{\F'}\I\{\rho(K)>0,\rho(K)\in\cdot\}\,\Lambda(dK).
\end{align*}
Naturally, our concentration inequalities require the constant
\begin{equation}\label{intas2}
s_0:=\sup\Big\{s>0:\int_{\F'}\I\{\rho(K)> 1\}e^{s\rho(K)}\,\Lambda(dK)<\infty\Big\}
\end{equation}
to be positive.

The function $\phi\colon[0,\infty)\to\R$, defined by $\phi(z):= e^z-1-z$, 
plays an important role in the sequel.

\begin{Theorem}\label{booleth} Assume that \eqref{intas1} holds and that $s_0>0$,
where $s_0$ is given by \eqref{intas2}.
Then the Poisson functional $F=\rho(Z)$ satisfies
\begin{equation}\label{conc}
\P(F-\E[F]\ge r)\le\exp\Big(\inf_{0<s<s_0}\Big(\int \phi(s u)\,\nu^*(du)-sr\Big)\Big), \quad r>0.
\end{equation}
\end{Theorem}
\begin{proof} 
We write $\eta=\sum^{\eta(\F')}_{k=1}\delta_{Z_k}$. Let $n\in\N$.
Then
$\eta_n:=\sum^{\eta(\F')}_{k=1}\I\{\rho(Z_k)\le n\}\delta_{Z_k}$
is a Poisson process with intensity measure
$\Lambda_n(dK):=\I\{\rho(K)\le n\}\Lambda(dK)$.
Define $Z_n:=Z(\eta_n)$ and $F_n:=\rho(Z_n)$. 
We wish to apply Corollary \ref{maincor} to the pair $(\eta_n,F_n)$. 
We start by
checking the integrability properties of the Poisson functional
$F_n$.  
First, we obtain from Lemma \ref{l3.2} that
$$
\E\Big[\int (D_KF_n)^2\,\Lambda_n(dK)\Big]\le\int \I\{u\le n\}u^2\,\nu(du)
\le n\int u\,\nu(du)
$$
which is finite by \eqref{intas1}. Since $\E[F_n]<\infty$ the 
Poincar\'e inequality (see \cite[Exercise 18.2]{LastPenrose17})
shows that $\E[F^2]<\infty$.
Secondly, we have for each $s\ge 0$ that
\begin{align*}
\E[e^{2sF_n}]\le \E\Big[\exp\Big[2s\int \rho(K)\,\eta_n(dK)\Big]\Big]
=\exp\Big[\int\big(e^{2s\rho(K)}-1\big)\,\Lambda_n(dK)\Big],
\end{align*}
where we have used a well-known formula for Poisson processes (see e.g.\ \cite{LastPenrose17}). 
The integral in the exponent is dominated by a multiple (depending on $n$ and $s$) of
$\int \I\{u\le n\}u\,\nu(du)$ and hence finite.
Thirdly, we have
\begin{align*}
\E\Big[\int\big(D_Ke^{sF_n}\big)^2\,\Lambda_n(dK)\Big]
&=\int\E\big[\big(e^{sF_n}\big(e^{sD_KF_n}-1\big)\big)^2\big]\,\Lambda_n(dK)\\
&\le\E\big[e^{2sF_n}\big]\int\big(e^{s\rho(K)}-1\big)^2\,\Lambda_n(dK).
\end{align*}
This is finite, since
$\int\I\{u\le n\}\big(e^{su}-1\big)^2\,\nu(du)$ is bounded by a multiple
of $\int u\,\nu(du)$.

By Lemma \ref{l3.2} and Lemma \ref{difflem} (applied with $(\eta_n,F_n)$ in place of $(\eta,F)$), 
\begin{align*}
V_{F_n}(s)&\le\int_{\F'}\big(e^{s\rho(K)}-1\big)\int_\Y\I\{y\in K\}\int_0^1
\exp\big(-(1-t)\Lambda_n\big(\F'_y\big)\big)\,dt\,\rho(dy)\,\Lambda_n(dK)\\
&=\int_{\F'}\big(e^{s\rho(K)}-1\big)
\int_K\tau\big(\Lambda_n\big(\F'_y\big)\big)\,\rho(dy)\,\Lambda_n(dK)=:h_n(s).
\end{align*}
Let $\rho^*_n(K):=\int_K\tau\big(\Lambda_n\big(\F'_y\big)\big)\,\rho(dy)$.
Then we have
\begin{align*}
\int_0^sh_n(u)\,du &=\int_{\F'}\I\{\rho(K)>0\}\hs\rho^*_n(K)\int_0^s\big(e^{u\rho(K)}-1\big)\,du\,\Lambda_n(dK)\\
&=\int\I\{\rho(K)>0\}\hs\rho^*_n(K) \Big(\frac{e^{s\rho(K)}-1}{\rho(K)}-s\Big)\,\Lambda_n(dK)\\
&=\int\frac{\rho^*_n(K)}{\rho(K)} \phi(s\rho(K))\,\Lambda_n(dK).
\end{align*}
For each $r> 0$ we now obtain from \eqref{foreachs} that 
\begin{equation}\label{foreachs2}
\P(F_n-\E[F_n]\ge r)
\le\exp\Big(\int\I\{\rho(K)\le n\}\frac{\rho^*_n(K)}{\rho(K)} \phi(s\rho(K))\,\Lambda(dK)-sr\Big). 
\end{equation}

As $n\to\infty$ we have $Z_n\uparrow Z$ and hence $F_n\uparrow F$.
Monotone convergence implies $\E[F_n]\to\E[F]$. We now assume that $s\in(0,s_0)$
and assert that
\begin{align}\label{e3.13}
\lim_{n\to\infty}\int\I\{\rho(K)\le n\}\frac{\rho^*_n(K)}{\rho(K)} \phi(s\rho(K))\,\Lambda(dK)
=\int\frac{\rho^*(K)}{\rho(K)} \phi(s\rho(K))\,\Lambda(dK).
\end{align}
Indeed, we have for each $y\in\Y$ that $\lim_{n\to\infty}\Lambda_n(\F'_y)=\Lambda(\F'_y)$,
and since $\tau(\cdot)\le 1$ we obtain for each $K\in\F'$ from dominated convergence
that $\lim_{n\to\infty}\rho^*_n(K)=\rho^*(K)$.
Hence \eqref{e3.13} follows from dominated convergence once we have shown that
$\int\phi(s\rho(K))\,\Lambda(dK)$ is finite.
By assumption \eqref{intas2} it is sufficient to show that
\begin{align}\label{e3.14}
\int\I\{u\le 1\}\phi(su)\,\nu(du)<\infty.
\end{align}
For $u\in[0,1]$ the definition of $\phi$ implies that
$\phi(su)\le u\phi(s)$ and \eqref{e3.14} follows.
Let $\varepsilon>0$ such that $r-\varepsilon>0$. Fatou's Lemma shows that
\begin{align*}
\P(F-\E[F]>r-\varepsilon)&\le \liminf_{n\to\infty}\P(F_n-\E[F_n]>r-\varepsilon)\\
&\le\exp\Big(\int\frac{\rho^*(K)}{\rho(K)} \phi(s\rho(K))\,\Lambda(dK)-s(r-\varepsilon)\Big),
\end{align*}
where we have used \eqref{e3.13} and \eqref{foreachs2}
to obtain the second inequality. Letting $\varepsilon\to 0$, we obtain the asserted
concentration inequality \eqref{conc}.
\end{proof}

Theorem \ref{booleth} can be generalized to Lipschitz
functions of $F$. Recall that a function $T\colon [0,\infty)\to\R$
is Lipschitz with Lipschitz constant $c_T\ge 0$ if
$|T(u)-T(v)|\le c_T |u-v|$ for all $u,v\ge 0$.

\begin{Theorem}\label{boolethLip} Let the assumptions of
Theorem \ref{booleth} be satisfied and let 
$T\colon [0,\infty)\to\R$ be a Lipschitz function with Lipschitz constant $c_T>0$.
Then the Poisson functional $G:=T(\rho(Z))$ satisfies
\begin{align*}
\P(G-\E[G]\ge r)\le\exp\Big(\inf_{0<s<s_0/c_T}\Big(\int \phi(c_Ts u)\,\nu^*(du)-sr\Big)\Big), \quad r>0.
\end{align*}
\end{Theorem}
\begin{proof} We generalize the proof of Theorem \ref{booleth}. 
Let $G_n:=T(F_n)$, $n\in\N$. We first note that
\begin{align}\label{e2.52}
|D_KG_n|=|T(f(\eta_n+\delta_K)-T(f(\eta_n))|\le c_T D_KF_n,\quad K\in\F'.
\end{align}
Since $|G_n|\le |T(0)|+c_T F_n$, we can use the first part of the above 
proof to conclude that the pair $(G_n,\eta_n)$ satisfies the assumptions of
Corollary \ref{maincor} with $s_0/c_T$ in place of $s_0$.
Using \eqref{e2.52} and the inequality $|e^u-1|\le e^{|u|}-1$, $u\in\R$,
we now obtain for all $s\ge 0$ that
\begin{align*}
V_{G_n}(s)\le|V_{G_n}(s)|\le \int_{\F'}\big(e^{sc_T\rho(K)}-1\big)
\int_K c_T\tau\big(\Lambda_n\big(\F'_y\big)\big)\,\rho(dy)\,\Lambda_n(dK):=h_n(s).
\end{align*}
Since
\begin{align*}
\int_0^sh_n(u)\,du=\int\frac{\rho^*_n(K)}{\rho(K)} \phi(sc_T\rho(K))\,\Lambda_n(dK),
\end{align*}
we can finish the proof as before.
\end{proof}

In the remainder of the paper we shall work with 
Theorem \ref{booleth} and not with its more general version.
However, all results can be formulated for Lipschitz functions of 
$\rho(Z)$.

Define a function $h\colon[0,\infty)\to[0,\infty]$ by 
\begin{align}\label{e2.83}
h(s):=\int_0^\infty u(e^{su}-1)\,\nu^*(du),\quad s\ge 0.
\end{align}
If $\nu^*=0$, then $h\equiv 0$. Otherwise
$h$ is finite and strictly increasing on $[0,s_0)$. Let
$h^{-1}\colon[0,\infty)\to[0,\infty]$ denote the generalized inverse of $h$,
defined by 
\begin{align*}
h^{-1}(u):=\inf\{s\ge 0: h(s)\ge u\},\quad u\ge 0,
\end{align*}
where $\inf\emptyset:=\infty$. If $\nu^*=0$, then $h^{-1}\equiv\infty$ on
$(0,\infty)$. Otherwise $h^{-1}$ is strictly increasing and continuous
on $[0,h(s_0-))$, where $h(s_0-):=\lim_{s\uparrow s_0}h(s)$.

\begin{Theorem}\label{booleth2} Under the assumptions of Theorem \ref{booleth},
\begin{align*}
\P(F-\E[F]\ge r)\le\exp\Big(-\int^r_0 h^{-1}(u)\,du\Big), \quad r\in (0,h(s_0-)).
\end{align*}
\end{Theorem}
\begin{proof} If $\int \rho(K)\,\Lambda(dK)=0$, then
$F\equiv 0$ and the result is trivial.
Hence we can assume that $\int \rho(K)\,\Lambda(dK)>0$.
We next show that then $\nu^*(0,\infty)>0$. By definition of $\nu^*$ it is sufficient
to show for each $K\in\F'$  that $\rho^*(K)>0$ whenever $\rho(K)>0$.
Since $\tau>0$ on $[0,\infty)$, it is sufficient to 
prove that $\Lambda(\F'_y)<\infty$ for $\rho$-a.e.\ $y\in K$. But this follows from
\begin{align*}
\int_K \Lambda(\F'_y)\,\rho(dy)=\iint \I\{y\in L\}\I\{y\in K\}\,\Lambda(dL)\,\rho(dy)
=\int \rho(K\cap L)\,\Lambda(dL),
\end{align*}
which is finite by \eqref{intas1}.

Since $\nu^*(0,\infty)>0$ we obtain for each $s\in(0,s_0)$ that
\begin{align*}
\frac{d}{ds}\int\phi(su)\,\nu^*(du)=h(s)\in(0,\infty),\qquad
\frac{d}{ds}h(s)=\int u^2e^{su}\,\nu^*(du)\in(0,\infty).
\end{align*}
In view of Theorem \ref{booleth} the proof can now be finished as that of \cite[Theorem 1]{Hou}.
\end{proof}

\begin{Rem}\label{Rbetter}\rm Proposition 3.2 in \cite{Wu}
implies \eqref{conc} with $\nu$ instead of $\nu^*$.
Since $\nu^*\le\nu$, our result improves this inequality. 
The larger $y\mapsto \Lambda(\F'_y)$ the larger
the improvement. Recall from \eqref{volfrac} that
$\P(y\in Z)=1-\exp\big(-\Lambda\big(\F'_y\big)\big)$
is the probability that the point $y\in\Y$ is covered by $Z$.
Our concentration inequality takes into account these covering probabilities
and hence the overlapping of distinct grains.
\end{Rem}


In the sequel we use the function $\psi\colon [0,\infty)\to (-\infty,\infty]$, defined by
\begin{align}\label{epsi}
\psi(z):= 1-\frac1z(1+z)\log(1+z),\quad z>0,
\end{align}
and $\psi(0):=\infty$. We also define 
\begin{align*}
m_i:=\int_0^\infty u^i\,\nu^*(du),\quad i\in\{0,1,2\}.
\end{align*}
The proof of the following corollary of Theorem \ref{booleth2} 
is similiar to that of \cite[Corollary 1]{Hou}.

\begin{Cor}\label{corbounded} Assume that \eqref{intas1} holds and that
$\nu^*\ne 0$. Assume also that
there is some $a>0$ such that $\rho(K)\le a$ for $\Lambda$-a.e.\ $K\in\F'$.
Then we have for each $i\in\{0,1,2\}$ that
\begin{align*}
\P(F-\E[F]\ge r)
\le \exp\bigg[\frac{r}{a}\,\psi\Big(\frac{a^{i-1}r}{m_i}\Big)\bigg],
\quad r>0.
\end{align*}
\end{Cor} 
\begin{proof} We first note that $h(s_0-)=\infty$. This follows by $\nu^*\neq0$ once we have shown that $s_0=\infty$. To this end, let $s>0$. Then, we have $\int\I\{\rho(K)>1\}e^{s\rho(K)}\,\Lambda(dK)\le e^{sa}\int \rho(K)\,\Lambda(dK)$ which is finite by (\ref{intas1}).

Let $i\in\{0,1,2\}$. In the case $i=0$ we can assume
that $m_0=\nu^*([0,\infty))>0$. (Otherwise there is nothing to prove.)
It is easy to see that
$h(s)\le m_i\hs a^{1-i}(e^{sa}-1)$ for all $s>0$;
cf.\ the proof of \cite[Corollary 1]{Hou} for the case $i=2$.
Therefore
\begin{align*}
h^{-1}(r)\ge \frac{1}{a}\log\bigg(1+\frac{ra^{i-1}}{m_i}\bigg),\quad r>0.
\end{align*}
Since
\begin{align*}
-\int^r_0\log(1+zu)\,du=r\,\psi(zr),\quad r>0,
\end{align*}
for each $z>0$, we deduce the assertion from Theorem \ref{booleth2}.
\end{proof}

%


\begin{Ex}\label{eHausdorff}\rm
In this example we specialize the setting of this 
section to the case $\Y=\R^d$ for some fixed integer $d\in\N$. 
We set $\F:=\F(\R^d)$. We fix
$r\in [0,d]$ and let $\lambda_r$ denote the $r$-dimensional
Hausdorff measure on $\R^d$. Let $\mathcal{F}^r$ denote the
set of all $C\in\F$ such that $\lambda_r(C\cap\cdot)$  is a locally
finite measure on $\R^d$. By \cite[Corollary 2.1.5]{Zaehle82},
this is a measurable set, that is $\mathcal{F}^r\in\mathcal{B}(\F)$.
Let $W\subset\R^d$ be a compact set and define 
$\rho:=\lambda_r(W\cap\cdot)$.
By \cite[Theorem 2.1.3]{Zaehle82}, we have that
$K\mapsto \rho(K)$ is measurable on $\F^r$,
so that the pair $(\mathcal{F}^r,\rho)$ satisfies the general assumptions
of this section (with $\mathcal{F}'=\mathcal{F}^r$).
\end{Ex}

\section{Stationary Boolean models}\label{secBooleand}

In this section we consider the setting of Example \ref{eHausdorff}
in the case $r=d$.
We let $\eta$ be a Poisson process on the space $\F^d$ of all
closed sets $K\subset \R^d$ with $\lambda_d(K)<\infty$.
We assume here that the intensity measure $\Lambda$ of
$\eta$ is of the translation invariant form
\begin{align}\label{lambdarep}
\Lambda=\int_{\F^d}\int_{\R^d}\I\{K+x\in\cdot\}\,dx\,\Q(dK),
\end{align}
where $K+x:=\{y+x:y\in K\}$  and $\Q$ is a $\sigma$-finite measure on $\F^d$
satisfying
\begin{align}\label{e4.2}
0<\gamma_1:=\int\lambda_d(K)\,\Q(dK)<\infty.
\end{align}

\begin{Ex}\rm
Let $\Q_0$ be a probability measure on $\F^d$ satisfying
$\int\lambda_d(K)\,\Q_0(dK)<\infty$ and let $\rho_0$
be a measure on $(0,\infty)$ such that $\int _0^\infty r^{d}\,\rho_0(dr)<\infty$. 
Assume that
\begin{align*}
\Q=\iint \I\{rK\in\cdot\}\,\rho_0(dr)\,\Q_0(dK).
\end{align*}
Then
\begin{align*}
\int \lambda_d(K)\,\Q(dK)= \iint\lambda_d(rK)\,\rho_0(dr)\,\Q_0(dK)
=\int r^d \rho_0(dr)\int\lambda_d(K)\,\Q_0(dK)<\infty.
\end{align*}
\end{Ex}

We fix a closed set $W\subset\R^d$ with positive finite volume and 
derive concentration inequalities for the Poisson functional
$$
F=\lambda_d(Z\cap W),
$$
where $Z(\mu)$, $\mu\in\mathbf{N}$, is given by \eqref{Zdef} and the $\sigma$-finiteness of
$\Lambda$ will be checked below. We do this by applying
the results of the previous section in the case $\rho:=\lambda_d(W\cap \cdot)$.
Let
$$
p:=1-e^{-\gamma_1}.
$$
By \eqref{volfrac}, we have $p=\P(0\in Z)$. Moreover, Fubini's theorem and
\eqref{lambdastat} below imply that
\begin{align*}
\E[F]=p\lambda_d(W),
\end{align*}
so that $p$ is the {\em volume fraction} of $Z$.

\begin{Theorem}\label{clasboole} Assume that \eqref{e4.2} holds.
Then the Poisson functional $F=\lambda_d(Z\cap W)$ satisfies
$$
\P(F-\E[F]\ge r)
\le\exp\Big(\inf_{s>0}
\Big(\frac{p}{\gamma_1}\iint\phi(s\lambda_d((K+x)\cap W))\,dx\,\Q(dK)-sr\Big)\Big),\quad r>0.
$$
\end{Theorem}
\begin{proof} We wish to apply Theorem \ref{booleth} in the
case $\rho=\lambda_d(W\cap\cdot)$.

Set $\Q'(dK):=\I\{\lambda_d(K)>0\}\,\Q(dK)$ and 
$\Lambda'(dK):=\I\{\lambda_d(K)>0\}\,\Lambda(dK)$ and
let $\eta'$ be a Poisson process with intensity measure
$\Lambda'$. Then $\lambda_d(Z(\eta')\cap W)$ has the same distribution
as $\lambda_d(Z(\eta)\cap W)$. Hence we can assume without loss of
generality that $\lambda_d(K)>0$ for $\Q$-a.e.\ $K$.
In particular $\Lambda$ is then $\sigma$-finite.

For each Borel set $K\subset \R^d$ we have that
\begin{align*}
\int\lambda_d(W\cap (K+x))\,dx=\iint \I\{y\in K+x\}\I\{y\in W\}\,dy\,dx\\
=\iint \I\{y\in K\}\I\{y+x\in W\}\,dy\,dx.
\end{align*}
By Fubini's theorem and \eqref{lambdarep} we obtain that
\begin{align*}
\int\lambda_d(W\cap K)\,\Lambda(dK)=\lambda_d(W)\int \lambda_d(K)\,\Q(dK),
\end{align*}
so that \eqref{e4.2} implies assumption \eqref{intas1}.

Let $s>0$. Then
\begin{align*}
\iint&\I\{\lambda_d((K+x)\cap W)>1\}
e^{s\lambda_d((K+x)\cap W)}\,dx\,\Q(dK)\\
&\le e^{s\lambda_d(W)} \iint\lambda_d((K+x)\cap W)\,dx\,\Q(dK)\\
&=e^{s\lambda_d(W)}\lambda_d(W) \int\lambda_d(K)\,\Q(dK)<\infty.
\end{align*}
Therefore we have $s_0=\infty$, where $s_0$ is given by \eqref{intas2}.

As at \eqref{e3.4} we define $\F^d_x:=\{K\in\F^d:x\in K\}$ for $x\in\R^d$.
From \eqref{lambdarep} we obtain that
\begin{equation}\label{lambdastat}
\Lambda\big(\F^d_x\big)=\Lambda\big(\F^d_0\big)=\int\lambda_d(K)\,\Q(dK)=\gamma_1.
\end{equation}
Hence $\tau\big(\Lambda\big(\F^d_x\big)\big)=p/\gamma_1$. Therefore the measure
$\nu^*$ defined by \eqref{levyanalog} is given by
\begin{align}\label{e4.33}
\nu^*=\frac{p}{\gamma_1}\int \I\{\lambda_d(K\cap W)\in\cdot\}\,\Lambda(dK).
\end{align}
Hence Theorem \ref{booleth} implies the assertion.
\end{proof}

The right-hand side of the concentration inequality provided by Theorem \ref{clasboole} is of
a rather complicated form.  In the sequel we shall derive more explicit versions.
We use the function $\psi$ defined by \eqref{epsi} and the constant
\begin{align*}
\gamma_2:=\int\lambda_d(K)^2\,\Q(dK).
\end{align*}

\begin{Cor}\label{corbounded2}  Assume that \eqref{e4.2} holds and 
that $a>0$ is such that $\lambda_d(K\cap W)\le a$ for 
$\Lambda$-a.e.\ $K\in\F^d$. Then $F=\lambda_d(Z\cap W)$ satisfies
\begin{align}\label{e4.60}
\P(F-\E[F]\ge r)
&\le \exp\bigg[\frac{r}{a}\,\psi\bigg(\frac{r}{p\lambda_d(W)}\bigg)\bigg],
\quad r>0,\\\label{e4.62}
\P(F-\E[F]\ge r)
&\le \exp\bigg[\frac{r}{a}\,\psi\bigg(\frac{a\gamma_1r}{p\lambda_d(W)\gamma_2}\bigg)\bigg],
\quad r>0.
\end{align}
\end{Cor} 
\begin{proof}  
We can apply Corollary \ref{corbounded}.
Using \eqref{e4.33} and \eqref{lambdarep} we obtain that
\begin{align*}
m_1&=\frac{p}{\gamma_1}\iint \lambda_d((K+x)\cap W)\,dx\,\Q(dK)
=p\lambda_d(W).
\end{align*}
Inequality \eqref{e4.60} now follows from
the case $i=1$ of Corollary \ref{corbounded}.

Similarly  we obtain that
\begin{align*}
m_2&=\frac{p}{\gamma_1}\iint \lambda_d((K+x)\cap W)^2\,dx\,\Q(dK)\\
&\le \frac{p}{\gamma_1}\int\lambda_d(K)\int\lambda_d((K+x)\cap W)\,dx\,\Q(dK)
= \frac{p}{\gamma_1}\lambda_d(W)\int\lambda_d(K)^2\,\Q(dK).
\end{align*}
Since $\psi$ is decreasing, the inequality \eqref{e4.62} follows from
the case $i=2$ of Corollary \ref{corbounded}.
\end{proof}

%

\begin{Rem}\label{r4.5}\rm Suppose there exists $a>0$ such that
$\lambda_d(K)\le a$ for $\Q$-a.e.\ $K\in\F^d$. Then
\eqref{e4.62} is superior to  \eqref{e4.60}.
If there exist $\gamma>0$ and
$K_0\in\F^d$ with $\lambda_d(K_0)>0$ such that $\Q=\gamma\delta_{K_0}$, then
both inequalities yield
\begin{align}\label{e4.67}
\P(F-\E[F]\ge r)
\le \exp\bigg[\frac{r}{\lambda_d(K_0)}\,\psi\bigg(\frac{r}{p\lambda_d(W)}\bigg)\bigg],
\quad r>0.
\end{align}
\end{Rem}

\begin{Rem}\label{Rexactas}\rm
The estimate \eqref{e4.67} 
is quite sharp. This can be seen by a
comparison with the concentration of a Poisson distributed random
variable $X$. Lemma 1.2 in \cite{Pen} provides the bound
$\P(X-\E[X]\ge r)\le\exp\big(r\cdot\psi(r/\E[X])\big)$,
$r>0$. Furthermore, only a slight modification of the last bound
leads to the exact asymptotic 
$\P(X-\E[X]\ge r)\sim(2\pi(\E[X]+r))^{-1/2}\exp\big(r\cdot\psi(r/\E[X])\big)$, as
$r\to\infty$, see page 1225 in \cite{Hou}.
\end{Rem}

\begin{Rem}\rm Choosing $a=\lambda_d(W)$ in \eqref{e4.60} yields
\begin{align}\label{e4.677}
\P(F-\E[F]\ge r)
&\le \exp\bigg[\frac{r}{\lambda_d(W)}\,\psi\bigg(\frac{r}{p\lambda_d(W)}\bigg)\bigg],\quad r>0.
\end{align}
The advantage of this result is that it holds under the only assumption \eqref{e4.2}.
The disadvantage is the occurence of $\lambda_d(W)^{-1}$ as a factor of $r$ outside
the logarithmic term.
This is in contrast with the situation in Remark \ref{r4.5}.
\end{Rem}

If $\lambda_d(\cdot)$ is not essentially bounded w.r.t.\ $\Q$,
we need an exponential moment assumption on $\Q$ to improve 
\eqref{e4.677} at least partially.
Define a function
$\tilde{h}\colon[0,\infty)\to[0,\infty]$ by
\begin{align}\label{e4.58}
\tilde{h}(s):=\int \lambda_d(K)\big(e^{s\lambda_d(K)}-1)\,\Q(dK),\quad s\ge 0.
\end{align}

\begin{Cor}\label{c4.3} Assume that \eqref{e4.2} holds.
Then the Poisson functional $F=\lambda_d(Z\cap W)$ satisfies
\begin{align*}
\P(F-\E[F]\ge r)\le\exp\bigg(-\int^r_0 \tilde{h}^{-1}\bigg(\frac{\gamma_1u}{p\lambda_d(W)}\bigg)\,du\bigg), \quad r>0.
\end{align*}
\end{Cor}
\begin{proof} We wish to apply Theorem \ref{booleth2}. Recall the
definition \eqref{e2.83} of the function $h$. By \eqref{e4.33}
we have for each $s\ge 0$ that
\begin{align*}
h(s)
&=\frac{p}{\gamma_1}\iint\lambda_d((K+x)\cap W)\big(e^{s\lambda_d((K+x)\cap W)}-1\big)
\,dx\,\Q(dK)\\
&\le\frac{p}{\gamma_1}\iint\lambda_d((K+x)\cap W)\big(e^{s\lambda_d(K)}-1\big)
\,dx\,\Q(dK)
=\frac{p\lambda_d(W)}{\gamma_1}\tilde{h}(s).
\end{align*}
Hence we have for each $r>0$ that 
$h^{-1}(r)\ge \tilde{h}^{-1}(\gamma_1r/(p\lambda_d(W)))$, 
so that Theorem \ref{booleth2} and the identity $s_0=\infty$ imply the assertion
once we have shown that $\lim_{s\to\infty} h(s)=\infty$. But this
follows from $\nu^*((0,\infty))>0$ (a consequence of $\gamma_1>0$).
\end{proof}

We illustrate Corollary \ref{c4.3} with two examples. Let $c:=\gamma_1/(p\lambda_d(W))$.

\begin{Ex}\label{Egamma}\rm Assume that $\Q(\{K:\lambda_d(K)\in du\})\le\alpha u^{-1}e^{-\beta u}\,du$,
where $\alpha,\beta>0$. On the right-hand side we have here the L\'evy measure of the gamma
distribution with shape parameter $\alpha>0$ and rate parameter
$\beta>0$; see e.g.\ \cite[Example 15.6]{LastPenrose17}. For instance, this assumption is satisfied
with $\alpha=1/d$ 
if $\Q\le \int \I\{rK_0\in\cdot\}\,r^{-1}e^{-\beta r^d}\,dr$, where $K_0\in\F^d$ has
$\lambda_d(K_0)=1$. Let $s\in(0,\beta)$. A simple calculation shows that
$\tilde{h}(s)\le \alpha/(\beta-s)-\alpha/\beta$, so that
\begin{align*}
\tilde{h}^{-1}(r)\ge \beta-\frac{\alpha\beta}{\beta r+\alpha},\quad r>0.
\end{align*}
It follows that 
\begin{align*}
\int^r_0\tilde{h}^{-1}(u)\,du\ge\beta r-\alpha\log\bigg(1+\frac{\beta r}{\alpha}\bigg),\quad r>0.
\end{align*}
Therefore we obtain from Corollary \ref{c4.3}
\begin{align*}
\P(F-\E[F]\ge r)\le\exp\bigg(-\beta r+\frac{\alpha}{c}\log\bigg(1+\frac{\beta c r}{\alpha}\bigg)\bigg).
\end{align*}
\end{Ex}

\begin{Ex}\label{Egamma2}\rm Assume that 
  $\Q(\{K:\lambda_d(K)\in du\})\le\beta^\alpha/\Gamma(\alpha)
  u^{\alpha-1}e^{-\beta u}\,du$, where $\alpha,\beta>0$. On the
  right-hand side we have here the gamma distribution with shape
  parameter $\alpha$ and scale parameter $\beta$.  A similar
  calculation as in Example \ref{Egamma}, 
yields $\tilde{h}(s)\le \alpha\beta^\alpha/(\beta-s)^{\alpha+1}-\alpha/\beta$
for $s\in(0,\beta)$, so that
\begin{align*}
\tilde{h}^{-1}(r)\ge \beta-\bigg(\frac{\alpha\beta^{\alpha+1}}{\beta r+\alpha}\bigg)^{1/(\alpha+1)},\quad r>0,
\end{align*}
and finally, by Corollary \ref{c4.3},
\begin{align}\label{gambound}
  \P(F-\E[F]\ge r)\le\exp\bigg(-\beta
  r+\frac{\alpha+1}{c}\bigg(\bigg(\frac{\alpha+\beta c
    r}{\alpha}\bigg)^{\alpha/(\alpha+1)}-1\bigg)\bigg).
\end{align}
\end{Ex}

\begin{Rem}\rm 
Examples \ref{Egamma} and \ref{Egamma2} are geometrically
quite different. 
Assume in the second example that
$\Q(\{K:\lambda_d(K)\in du\})=\beta^\alpha/\Gamma(\alpha)
  u^{\alpha-1}e^{-\beta u}\,du$. Then each bounded set contains
a finite number of grain centers (at least under a weak regularity assumption
on $\Q$). Ignoring overlapping, each grain contributes
a gamma distributed volume. 
Assume in the first example that
$\Q(\{K:\lambda_d(K)\in du\})= u^{-1}e^{-\beta u}\,du$.
Then each measurable set $B\subset\R^d$ with $0<\lambda_d(B)<\infty$
contains infinitely many grain centers. However, the sum
of the volumes of balls centered in $B$
follows a gamma distribution
with scale parameter $\beta$ and shape parameter $\lambda_d(B)$; 
see \cite[Example 15.6]{LastPenrose17}.
Roughly speaking, $Z\cap W$ might be interpreted as a finite union of
random sets whose volumes are approximately gamma distributed.
This might explain that the leading terms in both concentration inequalities
are the same. 
\end{Rem}

Our bounds of Corollary \ref{corbounded2} improve significantly
Theorem 3 in \cite{Hei} which deals with the stationary Boolean model
in $\R^d$ and which assumes $\Q$ to be a probability measure. The tail
bound in \cite{Hei} is only of order $\exp\big(-\O(r)\big)$ and
therefore not able to reproduce the tail behaviour of the Poisson
distribution in the special setting of Remark \ref{r4.5}. Further, the
constants we use arise naturally from the model and are much less
involved than the ones in \cite{Hei}. Moreover, we do not require that
the moment-generating function of $\lambda_d(Z_0)$ exists but only
make the milder moment assumptions $\gamma_1<\infty$, respectively
$\gamma_2<\infty$.

We note that the general concentration inequalities derived in
\cite{BP} can be applied to some configurations of the stationary Boolean model in $\R^d$, too. 
At least in the case of bounded grains, this application already improves 
the correspondent result of \cite{Hei}. However, the functionals considered in \cite{BP} appear 
unable to incorporate the volume
fraction. To be more precise, in the setting of Corollary \ref{corbounded2}, 
the bound \eqref{e4.62} 
is superior to the result
\begin{align*}
\P(F-\E[F]\ge r)&\le\exp\bigg(\frac{r}{a}
\cdot\psi\bigg(\frac{ar}{\gamma_2 \lambda_d(W)}\bigg)\bigg),\quad r\ge0,
\end{align*}
obtained from Corollary 3.3 in \cite{BP} by the bound
\begin{align*}
\int(D_KF)^2\,\Lambda(dK)&\le\iint((K+x)\cap W)^2\,dx\,\Q(dK)\le\lambda_d(W)\gamma_2.
\end{align*}

Finally, we apply Theorem \ref{thmecke}.

\begin{Prop}\label{mcvol}
Assume that (\ref{e4.2}) holds. Then
\begin{equation}\label{e4.11}
  \P[F-\E[F]\ge r]\le\exp\bigg(-\frac{\gamma_1}{p\lambda_d(W)}
\bigg(r+\E[F]\log\bigg(\frac{\E[F]}{r+\E[F]}\bigg)\bigg)\bigg),\quad r>0.
\end{equation}
\end{Prop}

\begin{proof}
By Lemma \ref{difflem} and equation \eqref{lambdastat}, we have, for 
$\Lambda$-a.e.\ $K\in\F^d$,
\begin{equation*}
\int_0^1\int_\mathbf{N}D_Kf(\eta_t+\mu)\,\Pi_{(1-t)\Lambda}(d\mu)\,dt\le\rho(K\cap W)
\frac{p}{\gamma_1}\le\lambda_d(W)\frac{p}{\gamma_1}.
\end{equation*}
Using the properness of $\eta$, we also get the bound
\begin{equation*}
\int D_Kf(\eta-\delta_K)\,\eta(dK)=\sum_{K\in\hs\eta}D_Kf(\eta-\delta_K)
=\sum_{K\in\hs\eta}\lambda_d\Big((K\cap W)\setminus\bigcup_{L\in\hs\eta-\delta_K}L\Big)\le F.
\end{equation*}
The assertion now follows from Theorem \ref{thmecke}
using the same truncation method as in the proof of Theorem \ref{booleth}.
\end{proof}

We note that $\int D_Kf(\eta-\delta_K)\,\eta(dK)$ actually equals the volume of the set of points 
which are covered by exactly one grain. Thus, as the Mecke formula
allows us to employ the functional $\int D_Kf(\eta-\delta_K)\,\eta(dK)$, we are equipped with a finer
tool to respect the interplay between the grains of $Z$.

\begin{Ex}\label{Eexp}\rm Let $\Q(\{K:\lambda_d(K)\in du\})=\beta e^{-\beta u}\,du$, that is the volume of the typical grain is exponentially distributed. The larger $\beta$ (and therefore the smaller $p$) the better is the specific bound (\ref{gambound}) in comparison to the general bound (\ref{e4.11}). If $\beta>0.14$, i.e. $p<0.9992$, then (\ref{gambound}) outplays (\ref{e4.11}) uniformly. If $\beta<0.13$, it is the other way round. Between, (\ref{e4.11}) might be better only for small values of $r$. Comparing the more general bound (\ref{e4.677}) with (\ref{e4.11}), we see the same principle. The latter wins when $p$ is large.
\end{Ex}

\bigskip
\noindent {\bf Acknowledgment:} We wish to thank S. Bachmann and G. Peccati for discussing with us an early version of their paper \cite{BP}. This work was supported
by the German Research Foundation (DFG) through the research unit
``Geometry and Physics of Spatial Random System''
under the grant LA 965/6-2.


\begin{thebibliography}{99}
\addcontentsline{toc}{chapter}{Literaturverzeichnis}


\bibitem{BP}
Bachmann, S.\ and Peccati, G.\ (2016). Concentration bounds
  for geometric Poisson functionals: Logarithmic Sobolev inequalities
  revisited. {\em Electron.\ J.\ Probab.} {\bf 21}, 1-44.

\bibitem{BLM03}
Boucheron, S., Lugosi, G.\ and Massart,
  P.\ (2003). Concentration inequalities using the entropy method. {\em Ann.\ Probab.} {\bf 31}, 1583-1614.

\bibitem{BLM13}
Boucheron, S., Lugosi, G.\ and Massart, P.\ (2013).
{\em Concentration Inequalities: A Nonasymptotic Theory of Independence.} 
Oxford University Press.

\bibitem{Che}
Chernoff, H.\ (1952). A measure of asymptotic efficiency
  for tests of a hypothesis based on the sum of observations. {\em Ann. Math. Statist.} {\bf 23}, 493-507.

\bibitem{SKM}
Chiu, S.N., Stoyan, D., Kendall, W.S.\ and Mecke, J.\ (2013).
{\em Stochastic Geometry and its Applications.}
Third Edition, Wiley, Chichester.

\bibitem{ERS}
Eichelsbacher, P., Rai\v{c}, M.\ and Schreiber, T.\ (2015).
Moderate deviations for stabilizing functionals in
  geometric probability.
{\em Ann. Inst. Henri Poincar\'e Probab. Stat.} {\bf 51}, 89-128.

\bibitem{Hei}
Heinrich, L.\ (2005).
Large deviations of the empirical volume fraction for stationary Poisson grain models.
{\em Ann. Appl. Probab.} {\bf 15}, 392-420.

\bibitem{Hou} Houdr\'e, C.\ (2002).
Remarks on deviation inequalities for
  functions of infinitely divisible random vectors.
{\em Ann. Probab.} {\bf 30}, 1223-1237.

\bibitem{HP}
Houdr\'e, C.\ and Privault, N.\ (2002). Concentration and
deviation inequalities in infinite dimensions via covariance
representations. {\em Bernoulli} {\bf 8}, 697-720.

\bibitem{HLS} Hug, D., Last, G.\ and Schulte, M.\ (2016). Second order
  properties and central limit theorems for geometric functionals of
  Boolean models. {\em Ann. Appl. Probab.} {\bf 26}, 73-135.

\bibitem{Kal}
Kallenberg, O.\ (2002). {\em Foundations of Modern  Probability.} Second Edition, Springer, New York.

\bibitem{LP11}
Last, G.\ and Penrose, M.D.\ (2011).
Poisson process Fock space representation, chaos expansion and covariance
  inequalities. {\em Probab. Theory Related Fields} {\bf 150}, 663-690.

\bibitem{LastPenrose17}
Last, G.\ and Penrose, M.D.\ (2017).
{\em Lectures on the Poisson Process.} Cambridge University Press. To appear.
\url{http://www.math.kit.edu/stoch/~last/seite/lectures_on_the_poisson_process/de}

\bibitem{Mas}
Massart, P.\ (2000). About the constants in Talagrand's
  concentration inequalities for empirical processes. {\em Ann. Probab.}
  {\bf 28}, 863-884.

\bibitem{Mat} Matheron, G.\ (1975).
{\it Random Sets and Integral Geometry}. Wiley, London.

\bibitem{Pen} Penrose, M.D.\ (2003). {\em Random Geometric Graphs}. Oxford University Press, Oxford.

\bibitem{Pri} Privault, N.\ (2009). {\em Stochastic Analysis in
    Discrete and Continuous Settings: With Normal
    Martingales}. Springer, Heidelberg.

\bibitem{SW}
Schneider, R.\ and Weil, W.\ (2008). {\em Stochastic and Integral Geometry.} Springer, Berlin.

\bibitem{Wu}
Wu, L.\ (2000). A new modified logarithmic Sobolev
  inequality for Poisson point processes and several
  applications. {\em Probab. Theory Related Fields} {\bf 118},
  427-438.

\bibitem{Zaehle82}
Z\"ahle, M.\ (1982). Random processes of Hausdorff rectifiable
closed sets. {\em Math. Nachr.} {\bf 108}, 49-72.

\bibitem{Zaehle84}
Z\"ahle, U.\ (1984). Random fractals generated by random cutouts.
{\em Math. Nachr.} {\bf 116}, 27-52. 


\end{thebibliography}
\end{document}